\pdfoutput = 1

\documentclass[12pt,a4paper,reqno]{amsart}

\usepackage{amssymb, amsmath, amsthm, amsxtra, amscd}

\let\rm\normalfont

	\def\hbar{{\bar h}}

\def\bbN{\mathbb{N}}

\let\Phi=\Varphi
\let\phi=\varphi

\def\implies{\ifinner \Rightarrow \else \Longrightarrow \fi} 
\let\incl=\subseteq

\let\isom=\cong               
\def\maps{\ifinner \mapsto \else \longmapsto \fi}

\def\thm#1{thm.\hbox{} #1} 
\def\enu#1 {\par \noindent (#1) }  

\def\absv#1{\vert#1\vert}			

\def\set#1{\{#1\}}

\newmuskip\halfthinmuskip
\newtheorem{theorem}{Theorem} 

\newtheorem{corollary}[theorem]{Corollary}

\def\enum#1-#2{%
	\in\{\if\voidarg{#1}1,2\else #1\fi,\ldots,#2\}%
}

\protected\def\voidarg#1{%
	//\let\unlessb\empty\else\let\unlessb\unless\fi 
  \begingroup 
  	\aftergroup\expandafter\aftergroup\unlessb
  	\def\tempvar{#1}%
	\ifx\tempvar\empty \aftergroup \iftrueb 
	\else \aftergroup \iffalseb 
	\fi
  \endgroup
}  

\def\iffalseb{\iffalse}
\def\iftrueb{\iftrue}
\let\If=\if

\def\diffD{\mathop{\raise .05em\hbox{${\mathsurround 0pt\smallsetminus}$}}}%
\def\diffS{\mathop{\raise .085em\hbox{$\mathsurround 0pt\scriptstyle\smallsetminus$}}}%

\def\restr{\raise -2 pt \hbox{$\m@th |$}}

\def\Mtext#1{\quad\mbox{\rm #1}\quad}

\def\smallbreak{\par \ifdim\lastskip<\smallskipamount\removelastskip\penalty -50\smallskip\fi}
\def\medbreak{\par \ifdim\lastskip<\medskipamount\removelastskip\penalty -100\medskip\fi}
\def\bigbreak{\par \ifdim\lastskip<\bigskipamount\removelastskip\penalty -200\bigskip\fi}

\title{A NOTE ON THE PRODUCT OF THE CONJUGATES OF A POLYNOMIAL}

\begin{document}%

\baselineskip = 16pt plus 0.1pt

\setbox0 = \hbox{\rm \large Micha\"el Bensimhoun,}
\setbox1 = \hbox{\rm Jerusalem, July 2014 \vrule height 15pt width 0pt }
\author{%
\noindent\box0\\
\box1\\
}%

\def\Gal{\mathop{\mathrm{Gal}}}

\begin{abstract}
\small 
\baselineskip = 14pt
The theorem proved in this note, although elementary, is related to
a certain misconception.
If $K$ is a field, $f\in K[X]$ is separable and
irreducible over $K$, and $g$ is a polynomial dividing $f$, whose
coefficients lie in some finite Galois extension of $K$,  
it may seem natural to assert that the product of 
the conjugates of $g$ over $K[X]$ is $f$.
But this assertion is wrong, except in one particular case.
In this note, we make the relation between $K$, $f$, the product
of the conjugates of $g$, and the coefficient field
of $g$, precise. In particular, it is shown that the 
product of the conjugates of $g$ over $K[X]$ is equal to 
$f^n$, with $n\in \bbN$. 
\end{abstract}

\maketitle

{\small \noindent {\bf\sc Key Words: }
Conjugate, polynomial, product of conjugate polynomials
}
\bigbreak

\parskip 3pt plus 1pt minus 1pt
\parindent 0pt  

The aforementioned result, although too elementary to be new,
is not so easily found inside common resources:
it does not appear in the textbooks we have checked, like \cite{La} or \cite{St}, 
nor does it appear, to our knowledge, 
inside common resources like \emph{Wikipedia} or \emph{Mathwork}. 
It is related to a misconception, which sometimes occurs 
even in the work of experienced mathematicians\footnote{According to a personal communication of 
Prof.~M.~Jarden to the author, during the preparation of his thesis.}, 
according to which, the product of the conjugates 
of a divisor of an irreducible separable polynomial $f\in K[X]$, in an extension of $K$, is 
the polynomial $f$ itself.   
This assertion is wrong in general, but the following theorem holds.

\begin{theorem}\label{thm3} 
Let $K$ be a field, and $f\in K[X]$ be separable and
irreducible over $K$.
Assume that $g=a_0+a_1X+\cdots a_\nu X^\nu$ 
is a polynomial dividing $f$, whose
coefficients $a_i$ lie inside a finite extension of $K$.
Let $M$ be the splitting field of $f$,
$L=K(a_0,a_1,\ldots,a_\nu)$, and $\,G=\Gal(M/K)\isom \Gal(M[X]/K[X])$.
Let §m§ denote the number of distinct conjugates of §g§ over §K[X]§,
and assume that $g^{\sigma_1}$, $g^{\sigma_2}$, \dots, $g^{\sigma_m}$ is an
enumeration of these conjugates, with $\sigma_i\in G$ ($m\leq |{G}|$). 

(i) 
There holds: $m=[L:K]$;

(ii) If 
$
h=g^{\sigma_1}g^{\sigma_2}\cdots g^{\sigma_m},
$
then $h\in K[X]$ and 
$ h=cf^n$, with $$n=[L:K]\deg(g)\,/\deg(f)  \quad \hbox{and}\quad  c\in K.$$
\end{theorem}

\begin{proof}

(i)
If $\sigma$ and $\sigma'$ belong to $G$, 
then $g^{\sigma}=g^{\sigma'}$ if and 
only if $\sigma'{\sigma}^{-1}$ fixes $g$; this is 
possible if and only if $\sigma'{\sigma}^{-1}$ fixes all the
coefficients of $g$, that~is,
if $\sigma'{\sigma}^{-1}$ belongs to $\Gal(M/L)$.
This can be rephrased as follows:  $g^{\sigma}=g^{\sigma'}$ if and 
only if $\sigma'\in \Gal(M/L)\sigma$.
Hence, each coset of the form $\Gal(M/L)\sigma$, with $\sigma\in G$, 
corresponds to one and only one conjugate of $g$ over $K$.
There are $\absv{G}/\absv{\Gal(M/L)}$ such
cosets, therefore the number of conjugates of $g$ over $K$ is 
$$
m=\absv{G}/\absv{\Gal(M/L)}
=[M:K]/[M:L]=[L:K].
$$

(ii)
Given $\sigma\in G$, let us consider the polynomial
$$
h^\sigma=g^{\sigma_1\sigma}\cdots g^{\sigma_m\sigma}.
$$
Since $\sigma$ is bijective, 
the elements $g^{\sigma_i\sigma}$ are pairwise distinct whenever 
the index $i$ varies in $\set{1,\ldots m}$.
But they are obviously conjugates of $g$ over $K[X]$, hence
the  elements $g^{\sigma_i\sigma}$ are 
in fact all the conjugates of $g$ over $K[X]$.
In other words, every $\sigma\in G$ fixes the set
$\set{g^{\sigma_1},\ldots g^{\sigma_m}}$ of conjugates of $g$.
As a consequence, their product, the polynomial $h$, is fixed by $G$.  
Thus the coefficients of $h$ belong to $K$: $h\in K[X]$.

Now, since $g$ divides $f$, it is clear that $g^{\sigma_i}$ divides
$f^{\sigma_i}=f$ for every $i\enum-m$.
Hence $h$ divides $f^m$.
Let $n\in \bbN$ be the smallest number such that 
$h$ divides $f^n$ in $K[X]$.
There exists $h'\in K[X]$ 
such that $hh'=f^n$.
If $f$ would divide $h'$, then $f$ would cancel in both
sides of the equation, and $h$ would divide
$f^{n-1}$, contradicting the minimality of $n$ with
respect to this property.
Thus, $f$ does not divide $h'$.

The ring $K[X]$ is a unique factorization domain, and $f$ is irreducible
in $K[X]$ according to the hypothesis; hence $f$ is prime in $K[X]$.
Since $f$ does not divide $h'$ it follows from the above equation 
that $f^n$ divides $h$ in  $K[X]$, and there holds 
$$
\frac{h}{f^n}h'=1.
$$
As a consequence, both $h/f^n$ and $h'$ must belong to $K$.
Thus, $h=cf^n$, with $c\in K$.

Finally, since $h=cf^n$, 
$$
\deg(h)=n\deg(f).
$$
On the other hand, since $\deg(g^\sigma_i)=\deg(g)$ for 
every $i\enum-m$,
$$
\deg(h)=m\deg(g).
$$ 
Combining these equations leads to 
$$
n=m\frac{\deg(g)}{\deg(f)}=[L:K]\frac{\deg(g)}{\deg(f)}.
$$
\end{proof}

\begin{corollary} \label{corthm3}
With the same hypotheses as in \thm \ref{thm3}, 
assume furthermore that $g$ is irreducible over $L$ and
that $f$ has a primitive root $\theta$ 
(that~is, every other root of $f$ belongs to $K(\theta)$).
Then $h=f$.
\end{corollary}

\begin{proof}
We observe that every other root $\theta'$ of $f$ has the same degree than
$\theta$ over $K$.
Since $K(\theta')\incl K(\theta)$, it follows that
$K(\theta')=K(\theta)$. 
In other words, every root of $f$ is primitive.
So, we can assume w.l.g.\@ that $\theta$ is also a root of $g$.

Now, it is clear that $L\incl K(\theta)$, since the roots of 
$g$ span the coefficients of $g$,
hence 
$$
[L:K]\deg(g)=[L:K][L(\theta):L]=[L(\theta):K]
=[K(\theta):K]=\deg(f).
$$
By \thm \ref{thm3}, we conclude that $n=1$, that~is, $h=f$.
\end{proof}

A part of \thm \ref{thm3} is true in a more general 
context, as stated in the following theorem.

\begin{theorem}
Let $R_1$ and $R_2$ be integral domains, with $R_1\incl R_2$,
$K_1$ be the fraction field of $R_1$, and $K_2$ be the fraction
field of $R_2$. 
Assume that $\theta$ is a prime element of $R_1$,
and that $\theta'\in R_2$ divides $\theta$ in $R_2$.
Assume also that the extension $K_2/K_1$ is finite and separable, 
and that $N_{K_1(\theta')/K_1}(R_2)\incl R_1$.
Let $\theta'_1$, $\theta'_2$,\dots, $\theta'_m$ be the 
distinct conjugates of $\theta'$ over $K_1$ (with $\theta'_1=\theta'$), 
and
$$
\abovedisplayskip=5pt \belowdisplayskip=5pt
\Theta=\theta'_1\theta'_2\cdots \theta'_m
=N_{K_1(\theta')/K_1}(\theta').
$$
Then
$\Theta=\theta^n u$, where
$u$ is unit of $R_1$ and $n\leq [K_1(\theta'):K_1]$.
\end{theorem}

\begin{proof}
Let us set $L=K_1(\theta')$.
Since $\theta'$ divides $\theta$ in $R_2$,
$\theta=\theta'\nu$, with $\nu\in R_2$.
It follows from the well known properties of 
the norm that
$$
N_{L/K_1}(\theta)=\theta^{[L:K_1]}
=N_{L/K_1}(\theta')N_{L/K_1}(\nu)
=\Theta N_{L/K_1}(\nu). 
$$
Moreover, $N_{L/K_1}(R_2)\incl R_1$, hence $\Theta$ divides
$\theta^{[L:K]}$ in $R_1$.

Let $n\in \bbN$ be the smallest number such that 
$\Theta$ divides $\theta^n$ in $R_1$:
There exists $\Theta'\in R_1$ 
such that $\Theta\Theta'=\theta^n$.
If $\theta$ would divide $\theta'$, $\theta$ would cancel from both
sides of the equation, and $\Theta$ would divide
$\theta^{n-1}$, contradicting the minimality of $n$ with 
respect to this property. 
Hence  $\theta$ does not 
divide $\Theta'$.

Since $\theta$ is prime in $R_1$, 
$\theta^n$ must divides $\Theta$ in $R_1$, and there holds
$$
(\Theta/\theta^n)\Theta'=1.
$$
As a consequence, $\Theta/\theta^n$ is a unit of $R_1$,
or what is the same, 
$$
\Theta=\theta^nu, \Mtext{with $u$ unit of $R_1$.}
$$
\end{proof}

\bigskip
\bigskip

\bibliography{poly_conj}{}
\bibliographystyle{amsplain}
\end{document}


\end{document}